\documentclass[12pt,a4paper]{amsart}
\usepackage{newlfont,amsmath,mathrsfs,enumerate,amssymb,array,amscd,ulem,xcolor}
\usepackage{paralist}

\usepackage{hyperref}
\usepackage[nameinlink]{cleveref}

\usepackage{braket}

\newcommand{\rA}{{\mathrm{A}}}

\newcommand{\rC}{{\mathrm{C}}}

\newcommand{\rE}{{\mathrm{E}}}
\newcommand{\rF}{{\mathrm{F}}}
\newcommand{\rG}{{\mathrm{G}}}

\newcommand{\calR}{{\mathcal{R}}}

\newcommand{\Vmin}{\mathbf{V}}

\newcommand{\Res}{{\mathrm{Res}}}

\newcommand{\Sp}{{\mathrm{Sp}}}

\newcommand{\Sym}{{\mathrm{Sym}}}

\newcommand{\calH}{{\cal{H}}}

\newcommand{\bbC}{{\mathbb C}}

\newcommand{\bbR}{{\mathbb{R}}}

\newcommand{\bbZ}{{\mathbb Z}}

\newcommand{\rO}{{\mathrm{O}}}
\newcommand{\rU}{{\mathrm{U}}}

\newcommand{\Spin}{{\mathrm{Spin}}}
\newcommand{\SL}{{\mathrm{SL}}}
\newcommand{\GL}{{\mathrm{GL}}}
\newcommand{\SU}{{\mathrm{SU}}}
\newcommand{\SO}{{\mathrm{SO}}}

\newcommand{\fraka}{{\mathfrak{a}}}

\newcommand{\frakg}{{\mathfrak{g}}}

\newcommand{\frakm}{{\mathfrak{m}}}

\newcommand{\frakn}{{\mathfrak{n}}}

\newcommand{\frakk}{{\mathfrak{k}}}

\newcommand{\frakp}{{\mathfrak{p}}}
\newcommand{\frakq}{{\mathfrak{q}}}

\newcommand{\frakt}{{\mathfrak{t}}}

\newcommand{\bGamma}{\overline{{\mathcal{R}}}}

\newcommand{\bfA}{\mathbf{A}}

\newcommand{\bfC}{\mathbf{C}}
\newcommand{\bfD}{\mathbf{D}}
\newcommand{\bfE}{\mathbf{E}}
\newcommand{\bfF}{\mathbf{F}}
\newcommand{\bfG}{\mathbf{G}}

\newcommand{\rSp}{\mathrm{Sp}}

\newcommand{\wtGL}{\widetilde{\mathrm{GL}}}

\newcommand{\Hom}{{\mathrm{Hom}}}

\newtheorem{lemma}{Lemma}[section] 

\newtheorem{prop}[lemma]{Proposition}
\newtheorem{thm}[lemma]{Theorem}
\newtheorem{cor}[lemma]{Corollary}

\newtheorem*{claim*}{Claim}

\begin{document}


\title[exceptional theta correspondences]{Duality for spherical representations in exceptional theta correspondences}

\author{Hung Yean Loke and Gordan Savin} 

\address{Hung Yean Loke, Department of Mathematics, National
  University of Singapore, 21 Lower Kent Ridge Road, Singapore 119077}
\email{matlhy@nus.edu.sg}

\address{Gordan Savin, Department of Mathematics, University of Utah,
  Salt Lake City, UT 84112} \email{savin@math.utah.edu}

\begin{abstract} 
We study the exceptional theta correspondence for real groups 
obtained by restricting the minimal representation of the split
exceptional group of the type $\mathbf E_n$, to a split dual
pair where one member is the exceptional group of the type
$\bfG_2$. We prove that the correspondence gives a bijection between
spherical representations if $n=6,7$, and a slightly weaker statement
if $n=8$.
\end{abstract}

\maketitle

\normalem

\section{Introduction}

Let $H$ be the group of real points of a split, simply connected
algebraic group of the type $\bfE_n$, and let $K_H$ be a maximal
compact subgroup of $H$.  The group $H$ contains a split dual pair
$G\times G'$ where one group is of the type $\bfG_2$, while the other
is a simply connected group of the type $\bfA_2$, $\bfC_3$ and
$\bfF_4$, respectively, see \cite{LS15}.  We shall not fix $G$ or $G'$
to be of the type $\bfG_2$ in advance. The actual choice will depend
on the nature of arguments.  

Let $\frakg$ and $\frakg'$ be the Lie
algebras of $G$ and $G'$, respectively, and let $K$ and $K'$ be the
maximal compact subgroups of $G$ and $G'$, respectively.
Let $\Vmin$ be the Harish-Chandra module of the minimal representation
of $H$.  Let $V$ be an irreducible $(\frakg, K)$-module and $V'$
be an irreducible $(\frakg', K')$-module. We say that $V$ and
$V'$ correspond if $V\otimes V'$ is a quotient of $\Vmin$. It is known
that the infinitesimal character of $V$ determines the infinitesimal
character of $V'$ by \cite{HPS} and \cite{Li}.  We call this a  \emph{weak
duality}.

Let $V$ be an irreducible $(\frakg, K)$-module. Following
\cite{Howe}, there is a $(\frakg',K')$-module $\Theta(V)$ such that
\[
\Vmin/\cap_\phi \ker \phi \simeq \Theta(V) \otimes V 
\]
where the intersection is taken over all $(\frakg,K)$-module
homomorphisms $\phi \colon \Vmin \rightarrow V$.  The aim in the
theory of correspondences is to prove that $\Theta(V)$ is a finite
length $(\frakg',K')$-module with a unique irreducible quotient $V'$,
and then conversely, that $\Theta(V')$ is a finite length $(\frakg,
K)$-module with $V$ as a unique irreducible quotient.  We call this a
\emph{strong duality}.  The goal of this paper is to establish the
strong duality for spherical representations.  The minimal
representation $\Vmin$ is itself spherical, and let $v_0$ be a
non-zero spherical vector in $\Vmin$.  The strong duality for
spherical representations follows from the identities
\[ 
U(\frakg) \cdot v_0=\Vmin^{K'} \text{ and } U(\frakg') \cdot v_0= \Vmin^{K} 
\] 
where $U(\frakg)$ and $U(\frakg')$ are the enveloping
algebras.  In this paper we prove both of these identities for the
dual pairs in $\bfE_6$ and $\bfE_7$, and thus we establish the
strong duality for spherical representations in these two cases, but
only one for the dual pair in $\bfE_8$, see Corollaries \ref{CorTheta1} and \ref{CorTheta2}. 
Note that only one identity is sufficient to establish a multiplicity one statement, more
precisely, if $V$ and $V'$ are spherical, then $V\otimes V'$ can
appear as a quotient of $\Vmin$ with multiplicity at most one.

\smallskip 
We now briefly sketch our arguments. Assume that $G$ is the smaller
member of the dual pair.  Let $\mathfrak k \oplus \frakp$ be the
Cartan decomposition of $\frakg$. Let $S(\frakp)$ be the symmetric
algebra generated by the vector space $\frakp$.  We have an
isomorphism $U(\frakg) \cong S(\frakp) \otimes U(\mathfrak k)$ of
vector spaces.  Since $G$ is smaller, there exists $0<p<2$ such that
the restriction to $G$ of every matrix coefficient of $\Vmin$ is
contained in $L^{p}(G)$.  This allows us to prove that a generic
spherical representation $V$ of $G$ is a quotient of $\Vmin$ by
integrating the matrix coefficients of $V$ against the matrix
coefficients of $\Vmin$.  As a consequence, the map $X\mapsto X\cdot
v_0$ is an injection of $S(\frakp)$ into $\Vmin^{K'}$.  In fact, this
is true if we replace $\Vmin$ \ by any non-trivial, spherical, unitary
representation of $H$, since the matrix coefficients of any such
representation decay faster than the matrix coefficients of $\Vmin$,
as proved in \cite{LS}. The minimality of $\Vmin$ will assure that
this map is a bijection.  More precisely, let $\mathfrak k_H\oplus
\mathfrak p_H$ be the Cartan decomposition of $\mathfrak h$, and let
$\omega$ be the highest weight of the adjoint action of $\frakk_H$ on
$\frakp_H$.  Then the $K_H$-types of $\Vmin$ are, see \cite{BK},
\[
\Vmin=\bigoplus_{n=0}^{\infty}V(n\omega). 
\]
We prove that the dimension of $S^n(\frakp)$, the $n$-the symmetric
power of $\frakp$, is equal to the dimension of $V(n\omega)^{K'}$, the
space of $K'$-fixed vectors in $V(n\omega)$. It follows at once that
the map $X\mapsto X\cdot v_0$ is a bijection of $S(\frakp)$ and
$\Vmin^{K'}$. In particular, $U(\frakg) \cdot v_0=
\Vmin^{K'}$. Working from the other side, so $G'$ is the larger member
of the dual pair, we can show that every spherical representation $V'$
of $G'$, that weakly corresponds to a spherical representation $V$ of
$G$, is a quotient $U(\frakg') \cdot v_0$. This again reduces the
problem of verifying $U(\frakg') \cdot v_0=\Vmin^{K'}$ to a branching
problem, i.e.  determining the dimension of $V(n\omega)^{K'}$.
 
 \smallskip 
 
 The last two sections are devoted to the branching problems. We now
 assume that $G'$ is of the type $\bfG_2$. It turns out that computing
 the dimension of $V(n\omega)^K$ is surprisingly easy and follows from
 some, more or less known, branching rules.  On the other hand,
 computing the dimension of $V(n\omega)^{K'}$ is progressively more
 difficult, since $K'$ is the same in all three cases. We do not know
 how to compute the dimension of $V(n\omega)^{K'}$ if $H$ is of the
 type $\bfE_8$. In the other two cases the branching is derived from a
 branching rule for certain quaternionic representations
 \cite{LokeJFA}, using a see-saw dual pair argument.
 
 \smallskip 
 We need to add a small caveat. If $H$ is of the type $\mathbf E_6$
 then the strong duality holds only in the context of slightly larger
 groups. More precisely, let $\tilde H$ be the semi-direct product of
 $H$ and its group of outer automorphisms $\mathbb Z/2\mathbb Z$.  If
 $G\times G' $ is the dual pair in $H$ such that $G\cong \SL_3(\mathbb
 R)$, then the centralizer of $G'$ in $\tilde H$ is $\tilde G\cong
 \SL_3(\mathbb R)\rtimes \mathbb Z/2\mathbb Z$. We prove that the
 strong duality holds between spherical $(\mathfrak g, \tilde
 K)$-modules and spherical $(\mathfrak g', K')$-modules. A similar
 situation holds for classical dual pairs $\Sp_{2n}(\mathbb R) \times
 \rO(p,q)$, where the strong duality fails if $\rO(p,q)$ is replaced
 by $\SO(p,q)$.

\section{Geometry of exceptional groups} \label{G}

\subsection{}
Let $\mathfrak h$ be a split simple real Lie algebra of the type
$\bfE_n$.  Fix a maximal split Cartan subalgebra $\fraka_H \subset
\mathfrak h$.  Let $\Phi(\mathfrak h, \fraka_H)$ be the root system
arising from the adjoint action of $\mathfrak a_H$ on $\mathfrak h$.
Let $H$ be the group of real points of the simply connected Chevalley
group attached to $\mathfrak h$.  Let $K_H$ be the maximal compact
subgroup of $H$ given as the group of fixed points of a Chevalley
involution of $G$, i.e. one that acts as $-1$ on $\fraka_H$.  Fix a set
of positive roots $\Phi^+(\mathfrak h, \fraka_H)$.  Then any element
in $\fraka_H$ can be uniquely written as a sum $\sum_i t_i
\omega_i^{\vee}$ where $t_i\in\mathbb R$ and $\omega_1^{\vee}, \ldots
, \omega_n^{\vee}\in \fraka_H$ are the fundamental co-weights.  The
dominant cone $\fraka_H^+\subset \fraka_H$ is the set of  all $\sum_i
t_i \omega_i^{\vee}$ where $t_i\geq 0$.  Let $A_H=\exp(\fraka_H)
\subset H$. Then $H$ has a Cartan decomposition
 \[ 
 H= K A_H^+ K 
 \] 
 where $A_H^+=\exp(\fraka_H^+)$.

\subsection{} \label{sec:dualpairs}
We shall now construct the dual pair $G\times G' \subset H$ in two
ways. In the first construction $G'$ will be of the type~$\bfG_2$.  In
the second construction $G$ will of the type~$\bfG_2$. In each of
the two constructions $G$ will come equipped with a Cartan
decomposition $G=KA^+K$ such that $K \subset K_H$ and $A^+ \subset
A^+_H$.  These inclusions will be a tool to understand integrality
properties of matrix coefficients of $H$ when restricted to $G$.

Let $\fraka \subset \fraka_H$ be the subalgebra spanned by
the fundamental co-characters $\omega_i^{\vee}$ corresponding to the
black points in the following marked Dynkin diagrams:

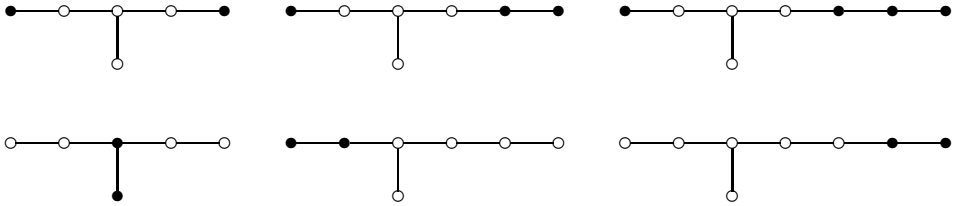
\begin{picture}(300,130)(-60,00)

\put(10,82){\line(0,1){16}}
\put(12,100){\line(1,0){16}}
\put(32,100){\line(1,0){16}}
%\put(52,100){\line(1,0){16}}
\put(-08,100){\line(1,0){16}}
\put(-28,100){\line(1,0){16}}

\put(-30,100){\circle*{4}}
\put(-10,100){\circle{4}}
\put(10,100){\circle{4}}
\put(10,80){\circle{4}}
\put(30,100){\circle{4}}
\put(50,100){\circle*{4}}
%\put(70,100){\circle{4}}

\put(115,82){\line(0,1){16}}
\put(117,100){\line(1,0){16}}
\put(137,100){\line(1,0){16}}
\put(157,100){\line(1,0){16}}
\put(97,100){\line(1,0){16}}
\put(77,100){\line(1,0){16}}

\put(75,100){\circle*{4}}
\put(95,100){\circle{4}}
\put(115,100){\circle{4}}
\put(115,80){\circle{4}}
\put(135,100){\circle{4}}
\put(155,100){\circle*{4}}
\put(175,100){\circle*{4}}

\put(240,82){\line(0,1){16}}
\put(202,100){\line(1,0){16}}
\put(222,100){\line(1,0){16}}
\put(242,100){\line(1,0){16}}
\put(262,100){\line(1,0){16}}
\put(282,100){\line(1,0){16}}
\put(302,100){\line(1,0){16}}

\put(200,100){\circle*{4}}
\put(220,100){\circle{4}}
\put(240,100){\circle{4}}
\put(240,80){\circle{4}}
\put(260,100){\circle{4}}
\put(280,100){\circle*{4}}
\put(300,100){\circle*{4}}
\put(320,100){\circle*{4}}

%\put(70,120){\circle*{4}}
%\put(130,120){\circle*{4}}

\put(10,32){\line(0,1){16}}
\put(12,50){\line(1,0){16}}
\put(32,50){\line(1,0){16}}
%\put(52,100){\line(1,0){16}}
\put(-08,50){\line(1,0){16}}
\put(-28,50){\line(1,0){16}}

\put(-30,50){\circle{4}}
\put(-10,50){\circle{4}}
\put(10,50){\circle*{4}}
\put(10,30){\circle*{4}}
\put(30,50){\circle{4}}
\put(50,50){\circle{4}}
%\put(70,100){\circle{4}}

\put(115,32){\line(0,1){16}}
\put(117,50){\line(1,0){16}}
\put(137,50){\line(1,0){16}}
\put(157,50){\line(1,0){16}}
\put(97,50){\line(1,0){16}}
\put(77,50){\line(1,0){16}}

\put(75,50){\circle*{4}}
\put(95,50){\circle*{4}}
\put(115,50){\circle{4}}
\put(115,30){\circle{4}}
\put(135,50){\circle{4}}
\put(155,50){\circle{4}}
\put(175,50){\circle{4}}

\put(240,32){\line(0,1){16}}
\put(202,50){\line(1,0){16}}
\put(222,50){\line(1,0){16}}
\put(242,50){\line(1,0){16}}
\put(262,50){\line(1,0){16}}
\put(282,50){\line(1,0){16}}
\put(302,50){\line(1,0){16}}

\put(200,50){\circle{4}}
\put(220,50){\circle{4}}
\put(240,50){\circle{4}}
\put(240,30){\circle{4}}
\put(260,50){\circle{4}}
\put(280,50){\circle{4}}
\put(300,50){\circle*{4}}
\put(320,50){\circle*{4}}
\end{picture}

Consider the first row of the marked diagrams. By inspection, one
checks that the adjoint action of $\fraka$ on the Lie algebra of
$\mathfrak h$ gives rise to a restricted root system $\Phi(\mathfrak
h, \fraka)$ of the type $\bfA_2$, $\bfC_3$ and $\mathbf F_4$,
respectively. The long root spaces are one-dimensional while the short
root spaces are 8-dimensional. (If the restricted root system is of
the type $\bfA_2$ then all roots are considered short.)  Let $C$ be
the derived group of the centralizer of $\fraka$ in $H$. Its Dynkin
diagram is obtained by removing the black vertices. Thus $C$ is a
simply connected Chevalley group of the type $\bfD_4$ in each of the
three cases.  The Chevalley involution of $H$ restricts to a Chevalley
involution of $C$.  The automorphism of order~3 of the root system
$\bfD_4$ can be lifted to an automorphism of $C$ commuting with the
Chevalley involution.  Let $G'$ be the group of fixed points in $C$ of
that automorphism. It is a Chevalley group of the type $\bfG_2$.  Let
$G$ be the centralizer of $G'$ in $H$. It is well known that $G$ is
split, simply connected, of the type $\bfA_2$, $\bfC_3$ and $\bfF_4$,
respectively. In fact, this is easy to see on the level of Lie
algebras, since the group $G'$ fixes a line in each short root space
for the restricted root system.  The maximal split torus in $\frakg$
is $\mathfrak a$. The restriction to $G$ of the Chevalley involution
of $H$ is a Chevalley involution of $G$ since it acts as $-1$ on
$\fraka$. Hence the set of fixed points in $G$ is a maximal compact
subgroup $K$ such that $K \subset K_H$.  The root system
$\Phi(\frakg,\fraka)$ is equal to the restricted root system
$\Phi(\mathfrak h,\fraka)$, and the choice of positive roots
$\Phi^+(\mathfrak h,\fraka)$ determines a choice of positive roots
$\Phi^+(\frakg,\fraka)$. In particular, all $\omega_i^{\vee}$'s
spanning $\fraka$ belong to the dominant cone $\fraka^+$ and it is
easy to check that they are in fact precisely the fundamental
co-weights for $\Phi^+(\frakg,\mathfrak a)$.  Hence $A^+\subset
A_H^+$.

\smallskip

Now consider the second row of the marked diagrams. In this case the
rank of $\fraka$ is two. By inspection, one checks that the adjoint
action of $\fraka$ on the Lie algebra of $\mathfrak h$ gives rise to a
restricted root system $\Phi(\mathfrak h, \fraka)$ of the type
$\bfG_2$ in each case. The long root spaces are one-dimensional while
the short root spaces are 9, 15 and 27-dimensional, respectively.  Let
$C$ be the derived group of the centralizer of $\fraka$ in $H$. Its
Dynkin diagram is obtained by removing the black vertices.  In each
case $C$ has an outer automorphism of order 2.  (In the case where $C$
is of the type $\bfA_2\times \bfA_2$ we pick the automorphism that, on
the level of the Dynkin diagrams, is a composite of the automorphism
of the $\bfE_6$ diagram followed by the automorphism of one of the two
$\bfA_2$ factors.)  Again, we pick the automorphism so that it
commutes with the Chevalley involution.  Let $G'$ be the group of
fixed points of the automorphism, and $G$ its centralizer in $H$.
Since $G'$ fixes a line in each short space it follows that $G$ is of
the type $\bfG_2$.  Arguing exactly as above we have $K \subset K_H$,
where $K$ is a maximal compact subgroup of $G$, and $A^+\subset
A^+_H$.

\section{Minimal representations}  

\subsection{}
We retain the setting of the previous section.  Recall that we
have fixed a Cartan decomposition $H=K_HA_H^+K_H$ where
$A_H^+=\exp(\mathfrak a_H^+)$ and $\mathfrak a_H^+$ is a cone whose
elements are sums $\sum_i t_i \omega_i^{\vee}$, where $t_i\geq 0$,
over all fundamental co-weights.  Let $||\cdot ||$ be an euclidean
norm on $\mathfrak a_H$. For any $\lambda \in \mathfrak a^*_H$, let
$a^{\lambda}$ be the character of $A_H$ defined by $a^{\lambda}=
\exp(\lambda(\log a))$, for all $a\in A_H$.  Let $\omega_i \in
\mathfrak a^*_H$ be the fundamental characters, and $\rho_H\in
\mathfrak a^*_H$ be the half sum of all positive roots of $H$.  Let
$\omega_b$ be the fundamental weight corresponding to the branching
vertex.  Let $\Vmin$ be the Harish-Chandra module of the minimal
representation of $H$.  The normalized leading exponent of $\Vmin$ is
$\rho_H -\omega_b$.  For our purposes, this means that the restriction
of any matrix coefficient of $\Vmin$ to $A_H^+$ is bounded by a
multiple of $(1 + ||\log a ||)^d a^{-\omega_b}$, for some integer $d$.
Let $p_H$ be the smallest positive real number such that
 \[ 
p_H \langle\omega_b , \omega_i^{\vee} \rangle \geq 2 \langle\rho_H ,
\omega_i^{\vee}\rangle
\] 
for all fundamental co-characters $\omega_i^{\vee}$.  Then
$a^{-\omega_b} \leq a^{-2\rho_H/p_H}$ for all $a\in A_H^+$ and, by
\cite[Theorem 8.48]{Knapp}, the matrix coefficients of $\Vmin$ are
contained in $L^{p_H+\epsilon}(H)$, for all $\epsilon >0$.  One easily
checks (see \cite{LS}) that $p_H$ is as in the following table:
\[
\begin{array}{c||c|c|c} 
H & \bfE_6 & \bfE_7 & \bfE_8 \\ \hline 
p_H &  8 & 9 & 29/3 \\ 
\end{array} 
\] 

\subsection{}
We now study integrability properties of the matrix coefficients of
$\Vmin$ restricted to~$G$.  We have a Cartan decomposition $G=K A^+
K$, where $A^+=\exp(\mathfrak a^+)$, where $\mathfrak a^+$ is a
subcone of $\mathfrak a_H^+$ whose elements are sums $\sum_i t_i
\omega_i^{\vee}$, where $t_i\geq 0$, over the fundamental co-weights
corresponding to the black vertices in the marked Dynkin
  diagrams in \Cref{sec:dualpairs}. Let $\rho\in \mathfrak a^*$ be
the half-sum of all positive roots of $G$.  Let $p$ be the smallest
positive real number such that
 \[ 
p \langle\omega_b , \omega_i^{\vee} \rangle \geq 2 \langle \rho,
\omega_i^{\vee}\rangle
\] 
for considering only $\omega_i^{\vee}$ of $G$, i.e.  those
corresponding to the black vertices in the marked Dynkin
diagrams. Then $a^{-\omega_b} \leq a^{-2\rho /p}$ for all $a\in A^+$
and the restriction to $G$ of a matrix coefficient of $\Vmin$ is
contained in $L^{p+\epsilon}(G)$, for all $\epsilon >0$.  For $G$
arising form the first family of marked Dynkin diagrams, i.e.  the
centralizer of $G$ is of the type $\bfG_2$, we have the following
values of $p$:
\[
\begin{array}{c||c|c|c} 
G & \bfA_2 & \bfC_3 & \bfF_4 \\ \hline 
p &  1 & 2 & 8/3 \\ 
\end{array}
\] 
For $G$ arising from the second family of marked Dynkin diagrams we
have the following values of $p$:
\[
\begin{array}{c||c|c|c} 
G & \bfG_2 & \bfG_2 & \bfG_2 \\ \hline 
p &  2 & 3/2 & 1 \\ 
\end{array}
\]

We record what we have found in a form that is useful for later. 

\begin{lemma} \label{L:integrability} 
Assume that $G \times G'$ is the dual pair in $H$ with $G$ smaller
than $G'$. Then there exists a positive real number $p<2$ such that
the restriction to $G$ of matrix coefficients of the minimal
representation of $H$ are contained in $L^{p+\epsilon}(G)$, for all
$\epsilon >0$. \qed
\end{lemma}

\subsection{} \label{sec:ppleseries}
Let $P=MAN$ the minimal parabolic subgroup of $G$ containing $A$ such
that the Lie algebra of the unipotent radical $N$ is spanned by the
positive root spaces.  Every $\lambda \in\mathfrak a_{\bbC}^*$
defines a character $a^\lambda = \lambda(\log(a))$ of $A$.  We can
extend this character to $P$, so that it is trivial on $MN$. The
normalized induction from $P$ to $G$ produces a smooth representation
$\pi_{\lambda}$ of $G$, naturally isomorphic to the space of smooth
functions on $M\backslash K$.  We have a $G$-invariant pairing
$(\cdot, \cdot )$ between $\pi_{-\lambda}$ and $\pi_{\lambda}$ given
by integrating functions over $K$.  Let $V_{\lambda}$ be the
Harish-Chandra module of $\pi_{\lambda}$.  Let $v_{\lambda}\in
V_{\lambda}$ be the unique $K$-invariant vector such that its value on
$M\backslash K$ is one.  Let
\begin{equation} \label{eqspherical}
\varphi_{\lambda}(g)=( \pi_{\lambda} (g)v_{\lambda},v_{-\lambda})
\end{equation}
 be the spherical function on $G$ corresponding to $\pi_{\lambda}$
 where, abusing the notation, $\pi_{\lambda}(g)$ denotes the group
 action.  It is well known that, for $\lambda\in i\mathfrak a^*$, the
 matrix coefficients of~$V_{\lambda}$, in particular
 $\varphi_{\lambda}$, are contained in $L^{2+\epsilon}(G)$, for all
 $\epsilon >0$.  Moreover if we fix a $q>2$, then \cite[Theorem
   8.48]{Knapp} implies that there exists a tubular
 neighborhood of $i\mathfrak a^*$ in $\mathfrak a^*_{\bbC}$, defined
 by the inequality $||\Re(\lambda )|| < \delta$ for some $\delta
 >0$ depending on $q$, such that, for $\lambda$ contained in the
 tubular neighborhood, the matrix coefficients of $V_{\lambda}$ are
 contained in $L^{q}(G)$.

\medskip 

Let $\Vmin^\vee$ be the contragradient of the minimal representation
$\Vmin$.  Let $(\cdot, \cdot )$ denote the natural pairing between
$\Vmin$ and $\Vmin^\vee$.  We shall fix, once for all a pair of
spherical vectors $v_0\in \Vmin$ and $\tilde v_0\in \Vmin^\vee$ of
spherical vectors such that $(v_0, \tilde v_0)=1$.  Let $\Phi(g)=
(\pi(g)v_0, \tilde v_0)$, $g\in H$, be the spherical function. Here
$\pi(g)$ denotes the group action of $H$ in a globalization of
$\Vmin$.

\begin{prop} \label{P:convergence} 
Assume that $G$ is the smaller member of the dual pair.  Then there is
a tubular neighborhood of $i\mathfrak a^*$ such that for every
$\lambda$ in the tubular neighborhood, the integral
\[ 
\langle v,u\rangle= \int_G ( \pi(g)v, \tilde v_0 ) (\pi_{\lambda} (g)
u, v_{-\lambda}) dg
\] 
is absolutely convergent for all $v$ in $\Vmin$ and $u\in
V_{\lambda}$.  The pairing $\langle v,u\rangle$ is $\frakg$-invariant,
i.e. for every $X \in \frakg$, $\langle \pi(X) v, u\rangle+
\langle v, \pi_{\lambda}(X) u\rangle =0$, and $\langle
v_0,v_{\lambda}\rangle \neq 0$ for all $\lambda$ in an open
neighborhood of $0$.
\end{prop}

\begin{proof} 
Since $G$ is smaller,  \Cref{L:integrability} states that the
restriction of the matrix coefficients of $V$ to $G$ is
$L^{p+\epsilon}$ for some $p<2$ and every $\epsilon > 0$.  Hence the
integral is convergent, by the H\"{o}lder's inequality, in a suitable
tubular neighborhood of $i\mathfrak a^*$, depending on $p$.

  Let $f(t)=(\pi_{\lambda}(g) \pi_{\lambda} (\exp(tX))u,
  v_{-\lambda})$, where $t\in \mathbb R$. Then
  \begin{equation} \label{eqXuv}
 \langle \pi(X) v, u\rangle =\int_G (\pi (g) v, \tilde v_0)
 \lim_{t\rightarrow 0} \frac{f(t)-f(0)}{t} dg.
 \end{equation}
The $\frakg$-invariance of the pairing is a formal consequence of
$G$-invariance of the measure~$dg$, provided we can switch the order
of integration and differentiation in \eqref{eqXuv}.  This will be
accomplished by the Lebesgue's Dominated Convergence Theorem by
finding an appropriate bound on $\frac{f(t)-f(0)}{t}$, uniform for all
small $t$.  By the mean value theorem, it suffices to find a bound for
$f'(t)=( \pi_{\lambda}(gg_{t}) w, v_{-\lambda})$ where $g_{t}=
\exp(tX)$ and $w=\pi(X) u$.  Note that $f'(0)$ is a matrix coefficient
of $V_{\lambda}$. Let $\mu =\Re(\lambda)$. By \cite[Proposition
    7.14]{Knapp}, there exists a constant $C_w$ such that
 \[ 
 | ( \pi_{\lambda}(g) w, v_{-\lambda})| \leq C_w \varphi_{\mu}(g) 
 \] 
 for all $g\in G$. We claim that there exists a constant $C>0$ such that 
 \[ 
 \varphi_{\mu}(gg_t) \leq C \varphi_{\mu} (g) 
 \] 
 for all $g\in G$ and $|t| \leq 1$. To that end, recall that $V_{\mu}$
 is realized as functions on $M\backslash K$, and $v_{\mu}$ is the
 constant function equal to one.  Let $C$ be such that, for all
 $|t|\leq 1$, the function $\pi_{\mu}(g_t) v_{\mu}$ is bounded by $C$
 on $M\backslash K$.  Since the pairing between $V_{\mu}$ and
 $V_{-\mu}$ is defined as integration of functions over $K$, it
 follows that
 \[ 
 \varphi_{\mu}(gg_t)= (\pi_{\mu}(gg_t) v_{\mu}, v_{-\mu}) =
 (\pi_{\mu}(g_t) v_{\mu}, \pi_{-\mu}(g^{-1} )v_{-\mu}) \leq C(
 v_{\mu}, \pi_{-\mu}(g^{-1}) v_{-\mu}) = C \varphi_{\mu} (g).
 \]
 This proves our claim.
Thus we have a uniform estimate of $|f'(t)|$, for $|t|\leq 1$, by
$C_w C \varphi_{\mu}$ and we can switch the order of integration and
 differentiation by the Lebesgue's Dominated Convergence Theorem.

\smallskip

Finally, to prove non-vanishing of the pairing, note that $\langle
v_0, v_{\lambda}\rangle $ is simply the spherical transform
\[ 
s(\lambda) = \int_G \Phi(g) \varphi_{\lambda} (g) dg. 
\] 
Using the Harish-Chandra estimates for $\varphi_{\lambda}$
\cite[Proposition 3]{A}, it follows, from the Lebesgue's Dominated
Convergence Theorem that $s(\lambda)$ is a continuous function on a
tubular neighborhood of $i\mathfrak a^*$.  (In fact, it is an analytic
function, see \cite{TV}, but we shall not need that.) Recall that
$\varphi_{\lambda} (g)>0$ for $\lambda$ real. For the same reason
$\Phi(g)>0$ since the parameter of the minimal representation is also
real. Thus $s(0)\neq 0$ since $s(0)$ is an integral of a positive,
analytic function.
 \end{proof} 

\Cref{P:convergence} implies that the matrix coefficient pairing
gives a family of $\frakg$-intertwining maps
\begin{equation} \label{eqmlambda}
m_{\lambda} \colon \Vmin \rightarrow V_{\lambda} 
\end{equation}
in the tubular neighborhood such that $m_{\lambda} (v_0) $ is a
non-zero multiple of $v_{\lambda}$ for $\lambda$ in a neighborhood of 0.

\section{Main Results}

\subsection{} \label{sec:VminKtypes}
Let $G\times G'$ be the dual pair in $H$ as before. We shall now give a
description of the maximal compact subgroups $K\times K' \subset
G\times G'$ and the $K_H$-types of the minimal representation
\begin{equation} \label{eqminKtypes}
\Vmin=\bigoplus_{n=0}^{\infty}V(n\omega) 
\end{equation}
where $V(n \omega)$ denotes the irreducible representation of $K_H$ of
the highest weight $n \omega$.  If we assume that $G'$ is of the type
$\bfG_2$, then $K'\cong \SU_2\times_2\SU_2$ while the rest of the
data is given by the following table:
\[
\begin{array}{c|c|c|c|c}
H & K_H & \omega & G & K \\ \hline 
\bfE_{6}  & \Sp_8& 
\omega_4 = (1,1,1,1) & \SL_3(\bbR) & \SO_3 \\
\bfE_{7} & \SU_8 & \omega_4 = (1,1,1,1,0,0,0,0) & \Sp_6(\bbR) & \rU_3 \\ 
\bfE_{8} & \Spin_{16} & \omega_8 = \frac{1}{2}(1,1,1,1,1,1,1,1) &
\bfF_{4} & \SU_2 \times_2 \Sp_6.
\end{array}
\]
Here $\omega_i$ denotes the $i$-th fundamental weight of $K_H$ as given
in \cite{Bou}.  We note that the groups listed in the second column are
in fact 2-fold covers of $K_H$, the maximal compact subgroup of $H$.
For notational convenience we shall often work with these 2-fold
covers, instead of $K_H$ proper.
  
In order to obtain the strong reciprocity statement in the case of the
dual pair in~$\bfE_6$, it is necessary to consider larger,
disconnected groups.  Let $\tilde H$ be a semi-direct product of $H$ and its
group of outer automorphisms  $\mathbb Z/2\mathbb Z$. The maximal compact subgroup
  $\tilde K_H$ of $\tilde H$ is isomorphic to $\Sp_8\times \mathbb
Z/2\mathbb Z$. We let the nontrivial element in $\mathbb Z/2\mathbb Z$
act on the type $V(n\omega)\subset \Vmin$ by $(-1)^n$. In this way
$\Vmin$ becomes a $(\mathfrak h, \tilde K_H)$-module. The centralizer
of $G'$ in $\tilde H$ is $\tilde G$, a semi-direct product of $G$ and $\mathbb Z/2\mathbb Z$. 
The maximal compact subgroup $\tilde K\subset \tilde G$ is isomorphic
to $\rO_3$.

\subsection{} 
Assume for the rest of the section that $G$ is the smaller member of
the dual pair. Using the symmetrizing map, $S(\frakp)$ can be
viewed as a vector subspace of $U(\frakg)$.  By the
Poincar\'{e}-Birkhoff-Witt theorem, multiplication of elements in
$S(\frakp)$ and $U(\mathfrak k)$ gives an isomorphism $U(\frakg)
\cong S(\frakp) \otimes U(\mathfrak k)$, compatible with the
adjoint action of $K$.  Let $\bbC$ be the trivial representation
of $U(\mathfrak k)$, and let
\[ 
E= U(\frakg)\otimes_{U(\mathfrak k)} \bbC\cong S(\frakp).
\] 
The center $Z(\frakg)$ maps into $E$ by $x \mapsto x \otimes 1$. Since
$G$ is split, this map is an isomorphism of $Z(\frakg)$ and $E^K\cong
S(\frakp)^K$.  We also have a decomposition $S(\frakp) \cong
\calH \otimes S(\frakp)^K$, where $\calH$ is an $K$-invariant
subspace of, so-called, harmonic polynomials.  Note that $E$ is
naturally a left $U(\frakg)$-module by left multiplication.
Considering $Z(\frakg)$ as a subalgebra of $U(\frakg)$, $E$ is also a
right $Z(\frakg)$-module.  Recall that $\lambda\in
\fraka_{\bbC}^*$ defines a one-dimensional character
$\chi_{\lambda}$ of $Z(\frakg)$ acting on the line $\mathbb
C_{\lambda}$.  Let
\[ 
E_{\lambda}= E \otimes_{Z(\frakg)} \bbC_{\lambda} \cong \calH.
\]  
Let $V_\lambda$ be the $(\frakg ,K)$-module of the principal series
  representation as in \Cref{sec:ppleseries}. 
By a Frobenius Reciprocity Theorem, we have a canonical map from
$E_{\lambda}$ to $V_{\lambda}$, sending $1\otimes 1$ to $v_{\lambda}$.
By \cite[Section 11.3.6]{RRGII}, this is an isomorphism for a generic
$\lambda$.

\begin{lemma} \label{L:injective} 
Let $X \in S(\frakp)$ be non-zero. Then $\pi_{\lambda}(X)\cdot
v_{\lambda}$ is non-zero for generic $\lambda \in \fraka^*_{\bbC}$.   
%We note that $\lambda$ depends on $X$.
\end{lemma} 

\begin{proof} 
Write $X=\sum X_i\otimes Y_i$ under the isomorphism
$S(\frakp)=\mathcal H\otimes S(\frakp)^K$, where the $X_i$'s
are linearly independent elements in $\calH$ and the $Y_i$'s
are non-zero.  Let $Z_i\in Z(\frakg)$ correspond to $Y_i$ under the
isomorphism of $S(\frakp)^K$ and $Z(\frakg)$. Let
$c_i=\chi_{\lambda}(Z_i)$. Note that $c_i\neq 0$ for a generic
$\lambda$.  Under the isomorphisms $V_{\lambda}\cong E_{\lambda}\cong
\calH$, the element $\pi_{\lambda}(X)\cdot v_{\lambda}\in V_{\lambda}$
corresponds to $\sum_i c_i X_i\in \calH$, and this is non-zero for a
generic $\lambda$.
\end{proof}

\begin{cor} \label{cor:injection}
The map $m \colon S(\frakp) \rightarrow \Vmin^{K'}$ given by $X
\mapsto X \cdot v_0$ is an injection.
\end{cor}

\begin{proof} 
Let $X$ be a non-zero element in $S(\frakp)$.  By \eqref{eqmlambda}
there is a family of $\frakg$-intertwining maps $m_{\lambda} \colon
\Vmin \rightarrow V_{\lambda}$ such that $m_{\lambda}(v_0)$ is a
non-zero multiple of $v_{\lambda}$ for all $\lambda$ in a neighborhood
$B$ of 0. Let $X \in S(\frakp)$ be non-zero.  It follows that
$m_\lambda(X \cdot v_0)$ is a non-zero multiple of $\pi_{\lambda}(X)
\cdot v_{\lambda}$ for all $\lambda \in B$.  By \Cref{L:injective},
$\pi_{\lambda}(X)\cdot v_{\lambda}$ is nonzero for some $\lambda$ in
$B$.  This proves that $X \cdot v_0 \neq 0$.
\end{proof}

\begin{prop} \label{prop:compareK}
The injection $m \colon S(\frakp) \rightarrow \Vmin^{K'}$ in the last
corollary is a bijection. In particular, $U(\frakg) \cdot v_0
= \Vmin^{K'}$.
\end{prop}

\begin{proof}  
Let $S^n(\frakp)$ be the $n$-th symmetric power of $\mathfrak
p$. For every non-negative integer $n$, the injection $m$ restricts to
an injection
\[ 
m_n \colon \oplus_{k=0}^n S^k(\frakp) \rightarrow \oplus_{k=0}^n
V(k\omega)^{K'}
\] 
In order to prove that $m$ is a bijection, it suffices to show that
each $m_n$ is a bijection, and this is a dimension check.  The
branching rule in \Cref{PKprimeinvariants}(i) and \Cref{PKinvariants}
implies that $\dim S^n(\frakp) = \dim V(n\omega)^{K'}$. This proves
the proposition.
\end{proof} 

We now derive several consequences of \Cref{prop:compareK}. Let
$S_{\lambda}$ be the irreducible spherical sub-quotient of
$V_{\lambda}$.  Recall that the isomorphism classes of irreducible
$K$-spherical $(\frakg ,K)$-modules correspond to the maximal ideals
in $Z(\frakg)$.

\begin{cor} 
For every $\lambda \in \fraka_{\bbC}^*$, $S_{\lambda}$ is a
quotient of $\Vmin$, i.e. $\Theta(S_{\lambda})\neq 0$.
\end{cor} 

\begin{proof} 
Note that $\Vmin^{K'}$ is a summand of $\Vmin$ and $\Vmin^{K'}\cong
U(\frakg) \otimes_{U(\mathfrak k)} \bbC$, as $\frakg$-modules.  The
corollary follows because $S_{\lambda}$ is a quotient of $U(\frakg)
\otimes_{U(\mathfrak k)} \bbC$. 
\end{proof}

We recall that $G$ is the smaller member of the dual pair.  Let
$Z(\frakg) = U(\frakg)^{G}$ be the center of $U(\frakg)$.  There
exists a homomorphism 
\begin{equation} \label{eqgamma}
\gamma \colon Z(\frakg')\rightarrow Z(\frakg)
\end{equation}
such that, for every $z\in Z(\frakg')$, $z = \gamma(z)$ on $\Vmin$.
The map $\gamma$ is surjective in this case if we replace $Z(\frakg)$
by $\tilde{Z}(\frakg) = U(\frakg)^{\tilde G}$.  After taking $\tilde
K$-invariants of both sides in \Cref{prop:compareK}, we get
\[ 
\tilde  Z(\frakg)\cdot v_0 = \Vmin^{\tilde K\times K'} = Z(\frakg') \cdot v_0. 
\]  
Let $\tilde S_{\lambda}$ be the irreducible $\tilde K$-spherical
$(\frakg ,K)$-module containing $S_{\lambda}$. The isomorphism classes
of irreducible $\tilde K$-spherical modules correspond to the maximal
ideals in $\tilde Z(\frakg)$.  Let $S_{\lambda'}$ be the irreducible
$K'$-spherical $(\frakg' ,K')$-module whose infinitesimal character is
the pullback, via $\gamma$, of the infinitesimal character of $\tilde
S_{\lambda}$.

\begin{cor}  \label{CorTheta1}
If $\Theta( S_{\lambda'})\neq 0$ then it is a finite length $(\frakg,
\tilde K)$-module with the unique irreducible quotient isomorphic to
$\tilde S_{\lambda}$.
\end{cor} 
\begin{proof} 
Since $\Theta( S_{\lambda'}) \otimes S_{\lambda'}$ is a quotient of
$\Vmin$, it follows that $\Theta( S_{\lambda'})$ is a quotient of
$\Vmin^{K'}$ and, therefore, a quotient of the cyclic $(\frakg, \tilde
K)$-module $U(\frakg)\otimes_{U(\mathfrak k)} \bbC$ where $\bbC$ is
the trivial representation of $\tilde K$. The module
$\Theta(S_{\lambda'})$ is also annihilated by the maximal ideal of
$\tilde Z(\frakg)$ corresponding to $\tilde S_{\lambda}$. By
  \cite[Corollary 3.4.7]{RRGI} $\Theta(S_{\lambda'})$ is admissible
  and by \cite[Theorem 4.2.1]{RRGI} it is of finite length.
 \end{proof}

\subsection{}
Now we want to go in the opposite direction, and show that $U(\frakg')
\cdot v_0 = \Vmin^{\tilde K}$.  Since the restriction of the matrix
coefficients of $\Vmin$ to $G'$ are not contained in $L^{2}(G')$, the
strategy we used for $G$ cannot be applied. However now we have
another way. Let
\[ 
U = U(\frakg') \cdot v_0\subseteq \Vmin^{\tilde K}.
\]  
Note that $U^{K'}=Z(\frakg') \cdot v_0$, and this is a direct
summand of $U$, considered a $Z(\frakg')$-module. (This is because $U$
is a direct sum of its $K'$-type subspaces, each of which is a
$Z(\frakg')$-submodule.)  Let $\chi_{\lambda}$ be a character of
$Z(\frakg)$ and let $\chi_{\lambda'}=\gamma^*(\chi_{\lambda})$ be the
character of $Z(\frakg')$ on a one-dimensional space
$\bbC_{\lambda'}$, obtained by pulling back by $\gamma$.  Then
$U_{\lambda'}=U\otimes_{Z(\frakg')} \mathbb C_{\lambda'}$ is a
quotient of $U$, generated by the unique $K'$-fixed line.  This proves
the following lemma.

\begin{lemma} \label{Pquotient} 
If $\chi_{\lambda'}=\gamma^*(\chi_{\lambda})$ then $S_{\lambda'}$ is a
quotient of $U$. \qed
\end{lemma}

We will now state the result corresponding to \Cref{prop:compareK}.
\begin{prop} 
Assume $H$ is of the type $\bfE_6$ or $\bfE_7$. Then $U(\frakg') \cdot
v_0 = \Vmin^{\tilde{K}}$.
\end{prop}
\begin{proof}
We need to discuss on a case by case basis.

\smallskip 
\noindent 
\underline{Case $\bfE_6$}: In this case the map $\gamma$ is a
bijection. Thus $S_{\lambda'}$ is irreducible for a generic $\lambda'$
and it is a quotient of $U$.  We can apply \Cref{L:injective} to $G'$
and, arguing as in the proof \Cref{cor:injection}, it follows that the
map $X \mapsto X\cdot v_0$ is an injection of $S(\frakp')$ into
$\Vmin^{\tilde K}$.  The branching rule in \Cref{PKinvariants} implies
that this map is also a bijection.

\smallskip

\noindent 
\underline{Case $\bfE_7$}: In this case the map $\gamma$ is a
surjection and the kernel is generated by a polynomial which under the
isomorphism $S(\frakp')^{K'}\cong Z(\frakg')$ corresponds to a
homogeneous polynomial $p$ of degree 4.  If $\chi_{\lambda'} =
\gamma^*(\chi_{\lambda})$ then by \cite[Theorem 6]{LSIsrael}
$S_{\lambda'}$ is irreducible for a generic $\lambda$ and, by
\Cref{Pquotient}, it is a quotient of $U$. An obvious alteration of
the proof of \Cref{cor:injection} implies that the map $X \mapsto
X\cdot v_0$ is an injection of $S(\frakp')/p S(\frakp')$ into
$\Vmin^{K}$ (in this case $K=\tilde K$).  The branching rule in
\Cref{PKprimeinvariants}(ii) implies that this is also a bijection.
\end{proof}

\begin{cor} \label{CorTheta2}
Assume $H$ is of the type $\bfE_6$ or $\bfE_7$.  Let $S_{\lambda'}$ be
an irreducible $K'$-spherical module whose infinitesimal character is
the pullback, via $\gamma$ in \eqref{eqgamma}, of the
  infinitesimal character of an irreducible $\tilde K$-spherical
  module $\tilde S_{\lambda}$. Then
\begin{itemize} 
\item $\Theta(S_{\lambda'})\neq 0$ and
\item $\Theta( \tilde S_{\lambda})$ is a finite length $(\frakg',
  K')$-module with the unique irreducible quotient isomorphic to
  $S_{\lambda'}$. \qed
\end{itemize} 
 \end{cor}

\section{Some simple branching rules}

The goal of this section is to compute the dimensions of
$V(n\omega)^{K}$ where $K$ is the maximal compact subgroup of the
centralizer of $\rG_2$, and $V(n\omega)$ are the types of the minimal
representation \eqref{eqminKtypes}.

\subsection{}
We start with three families of branching rules in a more general
setting than what we need, but the proofs are not more difficult.  All
groups are complex algebraic groups in this subsection unless
otherwise stated.

Let $R= \Sp_{2m}$ (respectively $\GL_{2m}$ and $\Spin_{4m}$).  Note
that $R$ acts naturally on its standard representation $V =\bbC^{2m}$
(respectively $\bbC^{2m}$, $\bbC^{4m}$).  Let $P$ be the maximal
parabolic subgroup of $R$ stabilizing an isotropic subspace of $V$ of
dimension $m$ (respectively $m$, $2m$). Let $P = MN$ be its Levi
decomposition. We tabulate $M$ and define a fundamental weight
$\omega$ of $R$ which defines a one-dimensional character of $P$.
\[
\begin{array}{c|c|c}
R &  \omega &  M  \\ \hline 
\Sp_{2m} & \omega_m = (1^m) &  \GL_m  \\
\GL_{2m} & \omega_m = (1^m,0^m) & \GL_m \times \GL_m  \\
 \Spin_{4m}  & \omega_{2m} = \frac{1}{2}(1^{2m}) &
{\wtGL_{2m}} =\GL_1 \times_m \SL_{2m} 
\end{array}
\]
Here $1^m$ denotes the $m$-tuple $(1,\ldots, 1)$ etc. Let $V(n\omega)$
be the the irreducible representation of $R$ with the highest weight
$n\omega$.  The next proposition describes the restriction of $V(n
\omega)$ from $R$ to~$M$. (The highest weights for $M$ are with
respect to the Borel subgroup of upper-triangular matrices.)

\begin{prop} \label{PBranching}
Let $n$ be a positive integer. We have
\begin{enumerate}[(i)]
\item 
\[
\Res^{\Sp_{2m}}_{\GL_m} V(n \omega_m) = \bigoplus_{\lambda} V(\lambda)
\]
where the sum is taken over all highest weights $\lambda = (\lambda_1,
\ldots, \lambda_m)$ of $\GL_m$ such that 
\[
n \geq \lambda_1 \geq \lambda_2 \geq \cdots \geq \lambda_m \geq -n
\]
and $\lambda_i \equiv n \pmod{2}$ for all $i$;

\item
\[
\Res^{\GL_{2m}}_{\GL_m \times \GL_m} V(n \omega_m) =
\bigoplus_{\lambda} V(\lambda) 
\]
where the sum is taken over all highest weights $\lambda = (\lambda_1,
\ldots, \lambda_{2m})$ of $\GL_m\times \GL_m$ such that 
\[
n \geq \lambda_1  \geq \cdots \geq \lambda_m \geq 0 
\]
and $\lambda_{m+i} = n- \lambda_{m-i+1}$ for all $i=1, \ldots, m$;
\item
\[
\Res^{\Spin_{4m}}_{\wtGL_{2m}} V(n \omega_{2m}) =
\bigoplus_{\lambda} V(\lambda) 
\]
where the sum is taken over all highest weights $\lambda =
(\lambda_1,\ldots, \lambda_{2m})$ of $\wtGL_{2m}$ such that
\[
\frac{n}{2} \geq \lambda_1 \geq \lambda_2 \geq \cdots \geq \lambda_{2m}
\geq -\frac{n}{2}
\]
$\lambda_{2i-1} = \lambda_{2i}$ for $i= 1, \ldots, m$ and $\lambda_j
\equiv \frac{n}{2} \pmod{\bbZ}$ for all $j = 1, \ldots, 2m$.
\end{enumerate}
\end{prop}

\begin{proof}
First we observe that a highest weight of $M$ is also a weight of
$R$.  It follows that the bounds on $\lambda_i$ hold. 

By the Borel-Weil theorem, we interpret $V(n \omega)$ as the space of
holomorphic sections of an $R$-equivariant line bundle on
$R/P$. We also interpret an irreducible representation~$V(\lambda)$ of
$M$ with the highest weight $\lambda$ as the space of sections of an
$M$-equivariant line bundle on $M/B$ where $B$ is a Borel subgroup of
$M$.

\begin{claim*}
The Levi subgroup $M$ has an open orbit on $R/P \times M/B$.  
\end{claim*}

\begin{proof}
Let $\bar{P} = M \bar{N}$ be the opposite parabolic subgroup to $P$.
Then the image of $\bar{N}$ is an open subset of $R/P$. The Levi
subgroup $M$ acts on $\bar{N}$ with an open orbit. By choosing an
appropriate point $x$ in the open orbit, we may assume that the
stabilizer $S$ of $x$ in $M$ is isomorphic to $\rO_m$, $\triangle
\GL_m$ and $\Sp(2m)$ respectively in the three cases. It is well known
that $S$ has an open orbit in $M/B$. This proves our claim.
\end{proof}

Pick $B$ in a standard position, i.e. containing the torus $T$ of
diagonal matrices in $M$ such that $S\cdot B/B$ is an open orbit in
$M/B$.  The claim implies that the restriction of $V(n \omega)$ to $M$
is multiplicity free. Moreover, only $V(\lambda)$ such that $\lambda =
n\omega$ on $S\cap T$ can appear in the restriction.  We will compute
$S \cap T$ in each case, and then we will derive a necessary condition
on $\lambda$ such that $V(\lambda) \subset V(n \omega)$:
\begin{itemize}
\item If $R = \Sp(2m)$ then $S\cap T = \{ t\in \GL_m : t \text{
  diagonal and } t^2=1 \}$. Hence
\[ 
\lambda_i \equiv n \pmod{2} \text{ for } i=1, \ldots, m.
\] 

\item If $R = \GL_{2m}$ then $S\cap T = \{ (t,t) \in \GL_m^2 : t
  \text{ diagonal} \}$. Hence
\[ 
\lambda_{m+i} = n -  \lambda_{m-i+1} \text{ for } i=1, \ldots, m.
\] 

\item If  $R = \Spin_{4m}$ then $S\cap T = \{ (t_1,t_1^{-1}, t_2,t_2^{-1},
  \ldots t_m,t_m^{-1}) : t_i \in \bbC \}$. Hence 
  \[ \lambda_{2i-i} =
  \lambda_{2i}\text{ for } i=1, \ldots, m.
  \] 
\end{itemize}
Summing up, we have shown that in all three cases of the proposition,
the left hand sides are contained in the respectively right hand
sides.

Finally we have to show the containments are equalities. We will prove
this by induction on $n$. If $n = 1$, the proposition holds by
checking the weights of $V(n \omega)$.  Suppose $n-1$ is true. Then we
have an $R$-equivariant map
\[
V((n-1)\omega) \otimes V(\omega) \rightarrow V(n \omega) 
\]
given by multiplying holomorphic sections. If $f \in V((n-1) \omega)$
and $f' \in V(\omega)$ are nonzero highest weight vectors of $M$ of
weights $\lambda$ and $\lambda'$ respectively, then the product
$f\cdot f' \in V(n \omega)$ is a nonzero highest weight vector of
weight $\lambda + \lambda'$. It follows by induction that all
irreducible representations of $M$ on the right hand sides occur in
$V(n \omega)$.  This completes the induction and proves the
proposition.
\end{proof}

\subsection{} 
Let $\tilde H$ denote the semi-direct product of $H$ and its group of outer automorphisms in all three cases, in particular, 
$H= \tilde H$ unless the type of $H$ is $\mathbf E_6$. Let  $\tilde K_H$ be the maximal compact subgroup of $\tilde H$.
The representations discussed in \Cref{PBranching}, for $m=4$, are
essentially the $\tilde K_H$-types of the minimal representations
 \eqref{eqminKtypes}.  Recall that, if $\tilde K_H\cong \Sp_8
\times (\bbZ/2 \bbZ)$, we extend $V(n \omega)$ to a representation of
$\tilde K_H$ such that the nontrivial element in $\bbZ/2 \bbZ$ acts on
by $(-1)^n$.  We assume now that $G'$ is of the type $\mathbf G_2$. Let $\tilde G$ be the centralizer of $G'$ in $\tilde H$.

\begin{prop}  \label{PKinvariants} 
Let $\tilde K$ be the maximal compact subgroup of $\tilde G$ and 
let $\frakk' \oplus \frakp'$ be the Cartan decomposition of the Lie algebra of $G'$, the split group of type $\mathbf G_2$. Then, for all $n$,
\[ 
\dim V(n \omega)^{\tilde K} =  \dim S^n(\frakp').
\] 
\end{prop}

\begin{proof}
 We define an $n$-string to be a sequence of integers 
\[ 
n \geq a_1 \geq a_2\geq a_3 \geq a_4 \geq a_5 \geq a_6 \geq a_7 \geq 0. 
\] 
The number of $n$-strings is equal to the dimension of the space of
homogeneous polynomials of degree $n$ in $8$ variables. Indeed, to
every $n$ string we can attach a monomial
\[ 
x_1^{n-a_1} x_2^{a_1-a_2} \cdot \ldots \cdot x_8^{a_7}
\] 
and vice versa. Since the dimension of $\frakp'$ is 8, the strategy of
the proof is to show that $\dim V(n \omega)^K$ is equal to the number
of $n$-strings.  We consider the three cases separately.

\medskip 
\noindent 
\underline{Case $\bfE_6$}: $\tilde K_H\cong \Sp_8 \times \langle \epsilon
\rangle$ and $\tilde{K} \cong  \rO_3$.  The embedding of $\tilde{K}$ into
$\tilde K_H$ is given by $x\mapsto (i(x), \det(x))$ where $i$ is the
composite of the embeddings
\[ 
 \rO_3 \subset \GL_3 \subset  \GL_4 \subset \Sp_8.
\] 
Since $\epsilon$ acts on $V(n \omega)$ by $(-1)^n$, our our task is to
determine the dimension of the subspace of $V(n\omega)$ on which
$i(\rO_3)$ acts by the character $\det^n$.  Let $W = V(a,b,c)$ be a
representation of $\GL_3$ with the highest weight $(a,b,c)$.  By the
Cartan-Helgason Theorem, $\dim W^{\SO_3} \leq 1$.  The representation
$W$ contains a nonzero $\SO_3$ invariant vector $v$ if and only if $a
\equiv b \equiv c \pmod{2}$. Moreover $\rO_3$ acts on $v$ by
$\det^{a+b+c}$.  We have thus shown that $W$ contains $\det_{\rO_3}^n$
with multiplicity at most one, and it is one if and only if $a \equiv
b \equiv c \equiv n \pmod{2}$.  Let $V = V(\lambda_1, \ldots,
\lambda_4)$ be the irreducible representation of $\GL_4$ of the
highest weight $\lambda_1 \geq \lambda_2 \geq \lambda_3 \geq
\lambda_4$.  By \Cref{PBranching}, $V(n \omega)$ is a multiplicity
free direct sum of $V$ such that $n \geq \lambda_i \geq -n$ and 
  $\lambda_i \equiv n \pmod{2}$.  By the Gelfand-Zetlin pattern,
 a summand $V$ contains $W$ if and only if
\[
n \geq \lambda_1 \geq a \geq \lambda_2 \geq b \geq \lambda_3 \geq c
\geq \lambda_4 \geq -n
\]
and when this happens, it contains $W$ with multiplicity one. Thus
each such string of integers, where all are odd or all are even,
contributes one dimension to $V(n \omega)^{\tilde K}$. Since these
strings are clearly in a bijection with the $n$-strings, we have
completed the proof in this case.

\medskip 
\noindent 
\underline{Case $\bfE_7$}: $K_H \cong \SL_8$ and $K \cong \GL_3$. The
embedding of $K$ into $K_H$ is given by the composite of the
embeddings
\[ 
   \triangle \GL_3\subset \GL_3^2 \subset \SL_4^2 \subset \SL_8.
\] 
An irreducible representation $W$ of $\GL_3^2$ contains a nonzero
$K$-invariant if and only if $W$ is of the form $V(a,b,c) \otimes
V(a,b,c)^*$.  Let $V$ be the restriction to $\SL_4^2$ of the
irreducible representation $V(\lambda_1, \ldots, \lambda_4) \otimes
V(\lambda_1, \ldots, \lambda_4)^*$ of $\GL_4^2$.  By
\Cref{PBranching}, $V(n\omega)$ is a multiplicity free direct
sum of such $V$ with $n\geq \lambda_i \geq 0$.  By the Gelfand-Zetlin
pattern, $V$ contains $W$ if and only if
\[
n \geq \lambda_1 \geq a \geq \lambda_2 \geq b \geq \lambda_3 \geq c
\geq \lambda_4 \geq 0.
\]
In other words, we have an $n$-string of integers.
Each such $n$-string contributes one dimension to
$V(n\omega)^K$. This completes the proof in this case.

\medskip 
\noindent 
\underline{Case $\bfE_8$}: $K_H \cong \Spin_{16}$ and $K \cong \SL_2
\times\Sp_6$. The embedding of $K$ into $K_H$ is given by the
composite of the embeddings
\[ 
 \SL_2 \times\Sp_6 \subset \wtGL_2 \times_2 \wtGL_6 \subset
\wtGL_8 \subset \Spin_{16}
\] 
 where the tildes above the groups denote double covers.  By the
 Cartan-Helgason theorem, an irreducible representation $W$ of
 $\wtGL_2 \times_2 \wtGL_6$ contains a nonzero $K$-invariant if and
 only if $W$ is of the form $V(d,d) \otimes V(a,a,b,b,c,c)$ where
 $a,b,c,d \in \frac{1}{2} \bbZ$, and $a,b,c$ are congruent to
   one another modulo $\bbZ$.  Let $V = V(\lambda_1, \lambda_1,
   \ldots, \lambda_4, \lambda_4)$ be the irreducible representation of
   $\wtGL_8$ where $\lambda_i$ are half-integers, congruent modulo
   $\mathbb Z$. By \Cref{PBranching}, $V(n \omega)$ is a multiplicity
   free direct sum of such $V$ with $n/2\geq \lambda_i \geq -n/2$ and
   $\lambda_i$ congruent to $n/2$ modulo $\mathbb Z$.  By the
   Gelfand-Zetlin pattern, a summand $V$ contains $W$ if and only if
   $d = \sum_{i=1}^4 \lambda_i - a - b -c$ and
\[
n/2 \geq \lambda_1 \geq a \geq \lambda_2 \geq b \geq \lambda_3 \geq c
\geq \lambda_4 \geq -n/2
\]
where all half-integers are congruent modulo ${\bbZ}$. When this
happens, $V$ contains $W$ with multiplicity one. Each such string of
half integers contributes one dimension to $V(n \omega)^K$. Since
these strings are in a bijection with the $n$-strings, we have
completed the proof in this case, as well.
\end{proof}

\section{A quaternionic see-saw}

The goal of this section is to compute the dimensions of $V(n
\omega)^{K'}$ where $K'$ is the maximal compact subgroup of $\rG_2$,
and $V(n\omega)$ are the types of the minimal representation  in
  \eqref{eqminKtypes}.  We shall compute these dimensions in the
case $H=\bfE_6$ and $H=\bfE_7$.  We do not know how to do this in the
case $H=\bfE_8$.

\subsection{}
Recall that the maximal compact of $\rG_{2,2}$ is $K'\cong \SU_{2,l}
\times \SU_{2,s}$.  Here the subscripts~$l$ and $s$ denote a long root
and a short root respectively. The embedding of $K'$ into $K_H$, for
all three $H$, is given by the sequence of embeddings
\[ 
 \SU_{2,l} \times \SU_{2,s} \subset \SU_2^4 \subset \Sp_8 \subset
\SU_8 \subset \Spin_{16}
\] 
where $\SU_{2,s}\subset \SU_2^3$ diagonally, and $\SU_2^4 \subset
\Sp_8$ is given by 4 orthogonal long roots.  The first and easy step
is to compute $\SU_{2,l}$-invariants in $V(n\omega)$.
\begin{inparaenum}[(i)]
\item
If $K_H\cong \Sp_8$,
then $V(n\omega)^{\SU_{2,l}}$ is naturally an $\Sp_6$-module. By the
branching rule  in \cite{WY}, it follows that
\[ 
V(n\omega)^{\SU_{2,l}}\cong V(n,n,0),  
\] 
the irreducible representation of $\Sp_6$ with the highest weight
$(n,n,0)$. 
\item If $K_H\cong \SU_8$, then $V(n\omega)^{\SU_{2,l}}$ is naturally
an $\SU_6$-module.  By the Littlewood-Richardson rule, it follows that
\[ 
V(n\omega)^{\SU_{2,l}}\cong \oplus_{a+b=n} V_{a,b}.
\] 
Here $V_{a,b}$ is the irreducible representation of $\SU_6$ with the
highest weight $a\omega_2 + b\omega_4$ where~$\omega_2$ and~$\omega_4$
are the second and the fourth fundamental weights of $\SU_6$
respectively.
\end{inparaenum}

\smallskip

The key result is the following lemma:

\begin{lemma} \label{PSU2invariants}
We have
\[
\dim V_{a,b}^{\SU_{2,s}} = \begin{pmatrix} a + 5
  \\ 5 \end{pmatrix} \begin{pmatrix} b + 5 \\ 5 \end{pmatrix}
-  \begin{pmatrix} a + 3 \\ 5 \end{pmatrix} \begin{pmatrix} b + 3
  \\ 5 \end{pmatrix}.
\]
\end{lemma}

In the above lemma if $0 \leq c \leq 1$, then we set $\begin{pmatrix}
  c + 3 \\ 5 \end{pmatrix} = 0$.  The proof of the lemma will be
  given in \Cref{sec:proofsu2}. It uses quaternionic representations
and a see-saw dual pair in the quaternionic group of type $\bfE_7$.
Before that, we shall derive our branching rules.

\begin{prop} \label{PKprimeinvariants} 
Let $K'$ be the maximal compact subgroup of $\rG_2$, and let
$\mathfrak k \oplus \frakp$ be the Cartan decomposition of the
centralizer $\frakg$ of $\rG_2$ in $\mathfrak h$. 
\begin{enumerate}[(i)] 
\item If $\mathfrak h$ is of the type $\bfE_6$ and $\frakg$ of the type
  $\rA_2$ then, for all $n$,
\[ 
\dim V(n\omega)^{K'} = \begin{pmatrix} n + 4
  \\ 4 \end{pmatrix} = \dim S^n(\frakp).
 \] 
\item If $\mathfrak h$ is of the type $\bfE_7$ and $\frakg$ of the type
  $\rC_3$ then, for all $n$,
\[ 
\dim V(n \omega)^{K'} = \begin{pmatrix} n+11 \\ 11 \end{pmatrix}
  - \begin{pmatrix} n+7 \\ 11 \end{pmatrix} = \dim
  S^n(\frakp)-\dim S^{n-4}(\frakp). 
  \] 
\end{enumerate}   
\end{prop}

\begin{proof}  We have $V(n\omega)^{\SU_{2,l}}= V(n,n,0)$. 
By \cite[(25.39) on p. 427]{FH}),  the restriction of $V_{n,0}$ to $\Sp_6$ decomposes as 
\[ 
V_{n,0}\cong\oplus_{c\leq n}
V(c,c,0).
\]  
 Hence $V(n,n,0) \cong V_{n,0} / V_{n-1,0}$, as $\SU_{2,s}$-modules, and by \Cref{PSU2invariants}, 
\[
\dim V(n\omega)^{K'} = \dim V_{n,0}^{\SU_{2,s}} - \dim
V_{n-1,0}^{\SU_{2,s}} = \begin{pmatrix} n + 5 \\ 5 \end{pmatrix}
- \begin{pmatrix} n + 4 \\ 5 \end{pmatrix} = \begin{pmatrix} n + 4
  \\ 4 \end{pmatrix}.
\]
Next we prove (ii).
We have $V(n\omega)^{\SU_{2,l}}\cong \oplus_{a+b=n} V_{a,b}$. 
Hence by \Cref{PSU2invariants},
\begin{align*}
\dim V(n \omega)^{K'} = & \sum_{a+b=n} V_{a,b}^{\SU_{2,s}} = 
\sum_{a+b=n}  \begin{pmatrix} a + 5
  \\ 5 \end{pmatrix} \begin{pmatrix} b + 5 \\ 5 \end{pmatrix}
- \sum_{a+b=n}  \begin{pmatrix} a + 3 \\ 5 \end{pmatrix} \begin{pmatrix} b + 3
  \\ 5 \end{pmatrix} \\
  = & \begin{pmatrix} n+11 \\ 11 \end{pmatrix}
  - \begin{pmatrix} n+7 \\ 11 \end{pmatrix}. 
\end{align*}
\end{proof}

\subsection{Quaternionic representations}
We recall some results from \cite{GW} and \cite{LokeJFA} on
quaternionic representations.  Let $G$ be a quaternionic simple,
simply connected real Lie group as in \cite{GW}, i.e. it has a maximal
compact subgroup $K$ isomorphic to $\SU_2 \times_2 M$.  Let $T \subset
\SU_2$ be a maximal Cartan subgroup. Let $\frakg$, $\frakt$ etc be the
complexified Lie algebras of $G$, $T$ etc, respectively.  The adjoint
action of $T$ gives a $\mathbb Z$-grading of $\frakg$. The grading
defines a theta-stable maximal parabolic subalgebra $\frakq =(\frakt
\oplus \frakm) \oplus \frakn$, where $\frakn$ is a two step nilpotent
Lie subalgebra with one-dimensional center $Z(\frakn)$. We identify
$T$ with $\rU_1$ (the group of complex numbers of norm one) so that
$z\in \rU_1$ acts by multiplication by $z$ on $\frakn/Z(\frakn)$ and
by multiplication by $z^2$ on $Z(\frakn)$.  Let $L = T \times_2
M$. Let $W$ be an irreducible finite dimensional representation of
$L$. The group $T\cong \rU_1$ acts on $W$ by multiplication by $z^k$
for some integer $k$. Hence we shall also use the notation $W=W[k]$ to
emphasize the action of $T$. In particular, $\bbC[k]$ will denote the
representation of $L$ trivial on $M$.  Extend $W$ trivially
over~$\frakn$ so that it is a representation of $\frakq$.  For $k \geq
2$, the (unnormalized) cohomologically induced representation
$\calR^1(W) =\Gamma^1_{K/L}(\Hom_{\frakq,L}(U(\frakg),W[k]))$ has
$\SU_2 \times_2 M$-types of the form
\[
\calR^1(W) = \bigoplus_{n=0}^\infty S^{k+n-2}(\bbC^2) \otimes
(\Sym^n(\frakn_M) \otimes W_M)
\]
where $\bbC^2$ is the standard, 2-dimensional representation of $\SU_2$,
$W_M$ is the restriction of $W$ to $M$ and $\frakn_M$ is the
restriction of $\frakn/Z(\frakn)$ to $M$.  Let $\bGamma^1(W)$ be the
unique irreducible subquotient $\calR^1(W)$ containing the lowest
$K$-type $S^{k-2}(\bbC^2) \otimes W_M$.

\smallskip 

Let $\rE_{7,4}$ be the simply connected quaternionic real form of type
$\bfE_7$, and let $\sigma_Z$ be its minimal representation constructed
in \cite{GW}.  Let $\rF_{4,4}$ be the split group of type~$\bfF_4$.
It is a quaternionic group, and its maximal compact subgroup is
isomorphic to $\SU_2 \times_2 \Sp_{6}$.  Let $\sigma_X$ be the
representation $\bGamma^1(\bbC[6])$ of $\rF_{4,4}$.  By
\cite[Prop. 8.5]{GW} we have an exact sequence of Harish-Chandra
  modules of $\rF_{4,4}$:
\begin{equation} \label{eqexactseq}
0 \rightarrow \sigma_X \rightarrow \calR^1(\bbC[6]) \rightarrow
\calR^1(\bbC[10]) \rightarrow 0.
\end{equation}

The group $\rE_{7,4}$ contains a see-saw of dual pairs \cite{LokeJFA} 
 
 \begin{picture}(100,82)(-130,10) 

\put(10,24){$\SU_{2,1}$} 

\put(75,24){$\SU_{2,s}$}

\put(37,36){\line(1,1){35}}

\put(10,74){$\rF_{4,4}$}
\put(37,71){\line(1,-1){35}}

\put(75,74){$\SU_6$}

\end{picture}

\noindent We shall restrict the minimal representation $\sigma_Z$ to
these two dual pairs. By \cite[Eqn. 6.6]{Lokethesis},
\begin{equation}  \label{eq:theta-one} 
\sigma_Z^{\SU_{2,s}} \cong \sigma_X 
\end{equation} 
as representations of $\rF_{4,4}$. The restriction of $\sigma_X$ to
$\SU_{2,1} \times \SU_6$ is given in \cite{LokeJFA}. In order to
describe the result, note that $\SU_{2,1}$ is also a quaternionic
group with the maximal compact subgroup isomorphic to $\SU_2 \times_2
\rU_1$, i.e. $M=\rU_1$ in this case. Irreducible representations of
$M=\rU_1$ are $\chi^n$, where $\chi(z)=z$ for all $z\in \rU_1$, and
$n\in\mathbb Z$.  If $a,b$ be are non-negative integers, then
$\calR^1(\chi^{a-b}[6+a+b])$ is a quaternionic discrete series
representation of $\SU_{2,1}$.  Recall that $V_{a,b}$ is the
irreducible representation of $\SU_6$ of the highest weight
$a\omega_2+b\omega_4$.

\begin{thm}  \cite[Theorem 6.1.2]{LokeJFA} 
 We have 
\[
\Res^{\rE_{7,4}}_{\SU_{2,1} \times \SU_6} \sigma_Z
=\bigoplus_{a,b=0}^\infty   \calR^1(\chi^{a-b}[6+a+b]) \otimes V_{a,b}. 
\]
\end{thm}

\subsection{Proof of \Cref{PSU2invariants}} \label{sec:proofsu2} 
It follows from this theorem and the see-saw pair that  
\[
 \dim V_{a,b}^{\SU_{2,s}} = \dim
 \Hom_{\SU_{2,1}}\left(\sigma_Z^{\SU_{2,s}},   \calR^1(\chi^{a-b}[6+a+b]) \right), 
\]
where $\sigma_Z^{\SU_{2,s}}\cong \sigma_X$ by  \eqref{eq:theta-one}. 

 Let $V_6 =S^2 \bbC^3$ be the 6 dimensional representation of $\SU_3$. 
 By the restriction formula in \cite{LokeJFA}, for $r \geq 6$, 
\[
\Res^{\rF_{4,4}}_{\SU_{2,1} \times \SU_3} \calR^1(\bbC[r]) = 
\bigoplus_{a,b = 0}^\infty \calR^1(\chi^{a-b}[r+a+b]) \otimes \left(\Sym^a V_6
\otimes \Sym^b(V_6^\vee) \right).
\]
 In particular, the above restriction is a direct sum of
  discrete series representation of $\SU_{2,1}$.
Combining the above for $r=6,10$ and \eqref{eqexactseq} we obtain  
\begin{align}
 \dim V_{a,b}^{\SU_{2,s}} & = \dim \left( \Sym^a V_6 \otimes
 \Sym^b(V_6^\vee) \right) - \dim \left( \Sym^{a-2} V_6 \otimes
 \Sym^{b-2}( V_6^\vee) \right) \label{eqh} \\ & = \begin{pmatrix} a +
   5 \\ 5 \end{pmatrix} \begin{pmatrix} b + 5 \\ 5 \end{pmatrix}
 - \begin{pmatrix} a + 3 \\ 5 \end{pmatrix} \begin{pmatrix} b + 3
   \\ 5 \end{pmatrix}. \nonumber
\end{align}
This proves \Cref{PSU2invariants}. \qed

\subsection*{Acknowledgement} 
This work was started during the second author's visit to the
Mathematical Sciences Institute of the National University of
Singapore in 2016, and was completed during the first author's visit
to the University of Utah in 2017. Both authors would like to thank
these institutions for hospitality and support. The first author was
supported in part by an MOE-NUS AcRF Tier 1 grant R-146-000-208-112.
The second author was supported in part by an NSF grant DMS-1359774.


\begin{thebibliography}{999}
\bibitem[A]{A} J.-P. Anker {\em The Spherical Fourier Transform of
  Rapidly Decreasing Functions. A Simple Proof of a Characterization
  due to Harish-Chandra, Helgason, Trombi, and Varadarajan}, Journal
  of Functional Analysis {\bf 96}, 331-349 (1991).

\bibitem[Bou]{Bou} N. Bourbaki, {\em Lie groups and Lie algebras},
  Chapters 4-6 (translated from the 1968 French original by Andrew
  Pressley), Elements of Mathematics, Springer-Verlag, (2002), ISBN
  3-540-42650-7.

\bibitem[BK]{BK} R. Brylinski and B. Kostant, {\em Minimal
  representation of $\rE_6$, $\rE_7$ and $\rE_8$ and the generalized
  Capelli Identity.} Proc. Natl. Acd. Sci. USA {\bf 91} (1994),
  2469-2472.

\bibitem[FH]{FH} W. Fulton and J. Harris, {\em Representation Theory:
  A First Course.} Graduate Texts in Mathematics {\bf 129},
  Springer-Verlag, (2004), ISBN 978-1-4612-0979-9.

\bibitem[GS]{GS} W. T. Gan and G. Savin, {\em Real and global
  endoscopic lifts from ${\mathrm{PGL}}(3)$ to $\rG_2$.} IMRN {\bf 50}
  (2003), 2699-€"2724.
  
\bibitem[GW]{GW} B. Gross and N. Wallach, {\em On quaternionic
  discrete series representations and their continuations}, J. reine
  angew. Math. {\bf 481} (1996), 73-123.

\bibitem[Ho]{Howe} R. Howe, {\em Transcending classical invariant theory.}
J. Amer. Math. Soc. 2 (1989), 535-552.

\bibitem[HPS]{HPS} J.-S. Huang, P. Pand\v{z}i\'{c} and G. Savin, {\em
    New dual pair correspondences.} Duke Math. J. {\bf 82}, No. 2
  (1996), 447-471.

\bibitem[Kn]{Knapp} A. W. Knapp , {\em Representation Theory of
  Semisimple Groups: An Overview Based on Examples.} Princeton
  University Press (2001).


\bibitem[Li]{Li} J.-S. Li, {\em The correspondences of infinitesimal
  characters for reductive dual pairs in simple Lie groups.} Duke
  Math. J. Volume 97, Number 2 (1999), 347-377.

\bibitem[Lo97]{Lokethesis} H. Y. Loke, PhD Thesis, Harvard University
  (1997).

\bibitem[Lo00]{LokeJFA} H. Y. Loke, {\em Restrictions of quaternionic
  representations.} Journal of Functional Analysis {\bf 172}, no. 2
  (2000), 377-403.

\bibitem[LS06]{LS} H. Y. Loke and G. Savin, {\em Rank and matrix
  coefficients for simply laced groups.} J. reine angew. Math. {\bf
  599} (2006), 201-216.

\bibitem[LS07]{LSIsrael} H. Y. Loke and G. Savin, {\em On local lifts
  from $G_2(R)$ to $Sp_6(R)$ and $F_4(R)$.} Israel Journal of
  Mathematics, 159, no. 1 (2007), 349-371.
  
  
\bibitem[LS15]{LS15} H. Y. Loke and G. Savin, {\em Rational forms of
  exceptional dual pairs.} J. of Algebra, {\bf 422} (2015) 683-696.

%\bibitem[MP]{MP} W. G. McKay and J. Patera, {\em Tables of Dimensions,
%  Indices and Branching Rules for Representations of Simple Lie
%  Algebras}. Lecture Notes in Pure and Applied Mathematics, Marcel
%  Dekker Inc. (1981).

\bibitem[TV]{TV} P. Trombi and V. S. Varadarajan, {\em Spherical transforms
  of semisimple Lie groups.} Annals of Mathematics, Second Series,
  {\bf 94} (1971), 246-€"303. 
  
%\bibitem[Va]{Va} V. S. Varadarajan, {\em Harmonic Analysis on Real
 % Reductive Groups.} Lecture Notes in Mathematics {\bf 576}, (1977).

\bibitem[WI]{RRGI} N. Wallach, {\em Real Reductive Groups I.} Pure and
  Applied Mathematics {\bf 132}, Academic Press, San Diego, (1988),
  ISBN-13: 978-0127329604.

\bibitem[WII]{RRGII} N. Wallach, {\em Real Reductive Groups II.}
Pure and Applied Mathematics {\bf 132-II},
Academic Press, San Diego, (1992), ISBN-13: 978-0127329611.

\bibitem[WY]{WY} N. Wallach and O. Yacobi, {\em A multiplicity formula
  for tensor products of SL2-modules and an explicit $\rSp(2n)$ to
  $\rSp(2n-2) \times \rSp(2)$ branching formula.} Contemporary
  Mathematics {\bf 190}, (2005), 151-155.
\end{thebibliography}
\end{document}